\documentclass[10pt]{article}
\usepackage{amssymb}
\usepackage{amsmath}
\usepackage{amsfonts}
\usepackage{amsthm}
\usepackage{graphics,graphicx}
\usepackage{float}
\usepackage[english]{babel}
\usepackage{authblk}

\newtheorem{theorem}{Theorem}
\newtheorem{lemma}{Lemma}

\newtheorem{corollary}{Corollary}

\title{Minimum supports of functions on the Hamming graphs with spectral constraints \thanks{ The reported study was funded by RFBR according to the research project 18-31-00126}}
\author[1,2]{Alexandr Valyuzhenich \thanks{graphkiper@mail.ru}}
\author[1,2]{Konstantin Vorob'ev \thanks{vorobev@math.nsc.ru}}

\affil[1]{Sobolev Institute of Mathematics, Novosibirsk, Russia}
\affil[2]{Novosibirsk State University, Novosibirsk, Russia}
\date{}

\begin{document}

\maketitle

\begin{abstract}
We study functions defined on the vertices of the Hamming graphs $H(n,q)$. The adjacency matrix of $H(n,q)$ has $n+1$ distinct eigenvalues
$n(q-1)-q\cdot i$  with corresponding eigenspaces $U_{i}(n,q)$ for $0\leq i\leq n$. In this work, we consider the problem of finding the minimum possible support (the number of nonzeros) of functions belonging to a direct sum $U_i(n,q)\oplus U_{i+1}(n,q)\oplus\ldots\oplus U_j(n,q)$ for $0\leq i\leq j\leq n$.
For the case $i+j\le n$ and $q\geq 3$ we find the minimum cardinality of the support of such functions and obtain a characterization of functions with the minimum cardinality of the support.
In the case $i+j>n$ and $q\geq 4$ we also find the minimum cardinality of the support of functions, and obtain a characterization of functions with the minimum cardinality of the support for $i=j$, $i>\frac{n}{2}$ and $q\geq 5$.
In particular, we characterize eigenfunctions from the eigenspace $U_{i}(n,q)$ with the minimum cardinality of the support for cases $i\le\frac{n}{2}$, $q\ge3$ and $i>\frac{n}{2}$, $q\ge 5$.

\end{abstract}

\section{Introduction}

Eigenspaces of graphs play an important role in algebraic graph theory (for example, see book \cite{Cvetkovic}). This work is devoted to some extremal properties of eigenspaces of the Hamming graphs. We consider the problem of finding the minimum cardinality of the support of eigenfunctions of the Hamming graph $H(n,q)$.
This problem is directly related to the problem of finding the minimum possible difference of
two combinatorial objects and to the problem of finding the minimum cardinality of the bitrades.
Bitrades are widely known subject and there are series of papers on Steiner bitrades \cite{HedayatKhosrovshahi}, bitrades in ordered sets \cite{Cho}, Latin bitrades \cite{Cavenagh,PotapovLatinBitrade} and bitrades in coding theory \cite{Ostergard,PotapovCardinalityOfAll}. In more details, connections between bitrades and eigenfunctions are described in
 \cite{KrotovExtendedtrades,KrotovTheminimumvolume,Krotovtezic,KrotovMogilnykhPotapov}.
The problem of finding the minimum size of the support of eigenfunctions was studied for the Johnson graphs in \cite{VMV}, for the Doob graphs in \cite{Bespalov}, for the cubic distance-regular graphs in \cite{Sotnikova} and for the Paley graphs in \cite{GoryainovKabanovShalaginovValyuzhenich}. The problem of finding the minimum cardinality of the support of eigenfunctions of the Hamming graphs $H(n,q)$ was completely solved for $q=2$ in \cite{Krotovtezic} based on ideas from \cite{Potapov}. In \cite{Valyuzhenich} this problem was solved for the second largest eigenvalue and arbitrary $q$.

In this work we find the minimum cardinality of the support of functions from the space $U_{[i,j]}(n,q)$ (a direct sum of eigenspaces of $H(n,q)$ corresponding to consecutive eigenvalues from $(q-1)n-qi$ to $(q-1)n-qj$) and give a characterization of functions with the minimum cardinality of the support for $i+j\le n$, $q\ge 3$ and for $i=j$, $i>\frac{n}{2}$, $q\ge 5$. In particular, we find the minimum cardinality of the support of eigenfunctions of the Hamming graphs $H(n,q)$.

The paper is organized as follows. In Section 2, we introduce basic definitions and notations. In Section 3, we define two families of functions that have the minimum size of the support in the space $U_{[i,j]}(n,q)$ for $i+j\le n$ and for $i+j>n$ respectively. In Section 4, we present auxiliary statements. In Section 5, we find the minimum size of the support of functions from the space  $U_{[i,j]}(n,q)$ for $i+j\le n$ and give a characterization of functions with the minimum cardinality of the support. In Section 6, we find the minimum size of the support of functions from the space $U_{[i,j]}(n,q)$ for $i+j>n$. In Section 7, we provide several curious examples and discuss further problems.

\section{Basic definitions}
Let $G=(V,E)$ be a graph. The set of neighbors of a vertex $x$ is denoted by $N(x)$. A real--valued function $f:V\longrightarrow{\mathbb{R}}$ is called a {\em $\lambda$--eigenfunction} of $G$ if the equality
$$\lambda\cdot f(x)=\sum_{y\in{N(x)}}f(y)$$ holds for any $x\in V$. Note that the vector of values of a $\lambda$--eigenfunction is an eigenvector of the adjacency matrix of $G$ with eigenvalue $\lambda$ or the all-zero vector. The set of all  $\lambda$--eigenfunctions of $G$ is called a {\em $\lambda$--eigenspace} of $G$.
The support of a real--valued function $f$ is the set of nonzeros of $f$. The cardinality of the support of $f$ is denoted by $|f|$.

Let $\Sigma_q=\{0,1,\ldots,q-1\}$. The {\em Hamming distance} $d(x,y)$ between vectors $x=(x_1,\ldots,x_n)$ and $y=(y_1,\ldots,y_n)$ from $\Sigma_{q}^n$ is the number of positions $i$ such that $x_i\neq{y_i}$.
The vertex set of the {\em Hamming graph} $H(n,q)$ is $\Sigma_{q}^n$ and two vertices are adjacent if the Hamming distance between them equals $1$.
It is well known that the set of eigenvalues of the adjacency matrix of $H(n,q)$ is $\{\lambda_{i}(n,q)=n(q-1)-q\cdot i\mid i=0,1,\ldots,n\}$.

Denote by $U_{i}(n,q)$ the $\lambda_{i}(n,q)$--eigenspace of $H(n,q)$. The direct sum of subspaces
$$U_i(n,q)\oplus U_{i+1}(n,q)\oplus\ldots\oplus U_j(n,q)$$ for $0\leq i\leq j\leq n$ is denoted by $U_{[i,j]}(n,q)$.

The {\em Cartesian product} $G\square H$ of graphs $G$ and $H$ is a graph with the vertex set $V(G)\times V(H)$; and
any two vertices $(u,u')$ and $(v,v')$ are adjacent if and only if either
$u=v$ and $u'$ is adjacent to $v'$ in $H$, or
$u'=v'$ and $u$ is adjacent to $v$ in $G$. Let $G=G_1\square G_2$, $f_1:V(G_1)\longrightarrow{\mathbb{R}}$ and $f_2:V(G_2)\longrightarrow{\mathbb{R}}$.
Define the {\em tensor product} $f_1\cdot f_2$ by the following rule: $(f_1\cdot f_2)(x,y)=f_1(x)f_2(y)$ for $(x,y)\in{V(G)}$.

Let us take two vectors $x=(x_1,\ldots,x_n)$ and  $y=(y_1,\ldots,y_m)$. The {\em tensor product} $x\otimes y$ is the vector $(x_1y_1,\ldots,x_1y_m,x_2y_1,\ldots,x_ny_m)$ of length $nm$.

Let $y=(y_1,\ldots,y_{n-1})$ be a vertex of $H(n-1,q)$, $k\in{\Sigma_q}$ and $r\in\{1,2,\ldots,n\}$. We consider the vector $x=(y_1,\ldots,y_{r-1},k,y_{r},\ldots,y_{n-1})$ of length $n$.
Given a function $f:\Sigma_{q}^{n}\longrightarrow{\mathbb{R}}$,  we define the function $f_{k}^{r}:\Sigma_{q}^{n-1}\longrightarrow{\mathbb{R}}$ by the rule $f_{k}^{r}(y)=f(x)$. We note that $f_{k}^{r}=f|_{x_r=k}$.

A function $f:\Sigma_{q}^n\longrightarrow{\mathbb{R}}$ is called {\em uniform} if for any $r\in\{1,2,\ldots,n\}$ there exists $l(r)\in\Sigma_q$ such that $f_{k}^{r}=f_{m}^{r}$ for all $k,m\in{\Sigma_q\setminus \{l(r)\}}$.

Recall that $S_n$ is the set of all permutations of length $n$. Let $f(x_1,x_2,\ldots,x_n)$ be a function and $\sigma\in{S_n}$.  We define the function $f_{\sigma}$ by the following rule: $f_{\sigma}(x)=f(x_{\sigma(1)},x_{\sigma(2)},\ldots,x_{\sigma(n)})$.
\section{Constructions of functions with the minimum size of the support}
We define the function $a_{1}(k,m):\Sigma_{q}^2\longrightarrow{\mathbb{R}}$ for $k,m\in{\Sigma_q}$ by the following rule:
$$
a_1(k,m)(x,y)=\begin{cases}
1,&\text{if $x=k$ and $y\neq m$;}\\
-1,&\text{if $y=m$ and $x\neq k$;}\\
0,&\text{otherwise.}
\end{cases}
$$
We note that $|a_1(k,m)|=2(q-1)$ for $k,m\in{\Sigma_q}$. The set of functions $a_1(k,m)$ where $k,m\in{\Sigma_q}$ is denoted by $A_1$.

We define the function $a_{2}(k,m):\Sigma_{q}\longrightarrow{\mathbb{R}}$ for $k,m\in{\Sigma_q}$ and $k\neq m$ by the rule:
$$
a_2(k,m)(x)=\begin{cases}
1,&\text{if $x=k$;}\\
-1,&\text{if $x=m$;}\\
0,&\text{otherwise.}
\end{cases}
$$
The set of functions $a_2(k,m)$ where $k,m\in{\Sigma_q}$ and $k\neq m$ is denoted by $A_2$.

Let $A_3=\{f:\Sigma_{q}\longrightarrow{\mathbb{R}}\mid f\equiv 1\}$. By the definition of an eigenfunction we see that $A_1\subset U_{1}(2,q)$, $A_2\subset U_{1}(1,q)$ and $A_3\subset U_{0}(1,q)$.

We define the function $a_{4}(m):\Sigma_{q}\longrightarrow{\mathbb{R}}$ for $m\in{\Sigma_q}$ by the rule:
$$
a_4(m)(x)=\begin{cases}
1,&\text{if $x=m$;}\\
0,&\text{otherwise.}
\end{cases}
$$
The set of functions $a_4(m)$ where $m\in{\Sigma_q}$ is denoted by $A_4$.
Note that $A_4\subset U_{[0,1]}(1,q)$.

 \begin{figure}[H]
 \begin{center}
\includegraphics[scale=0.45]{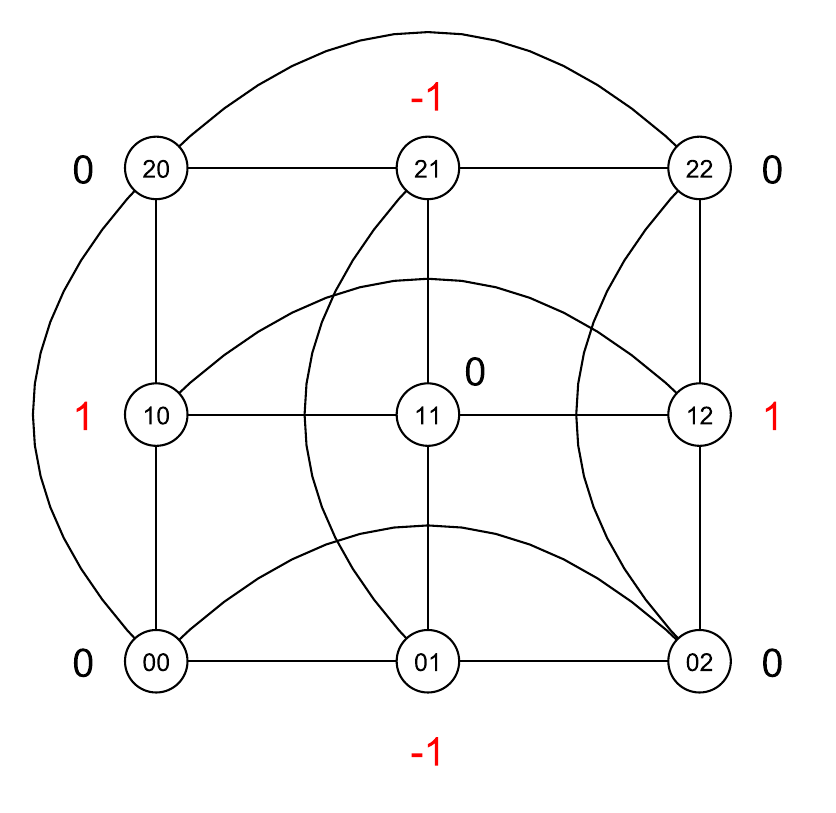}
\end{center}
\caption{Function $a_1(1,1)$ for $q=3$.}\label{A1}
\end{figure}

\begin{figure}[H]
 \begin{center}
\includegraphics[scale=0.35]{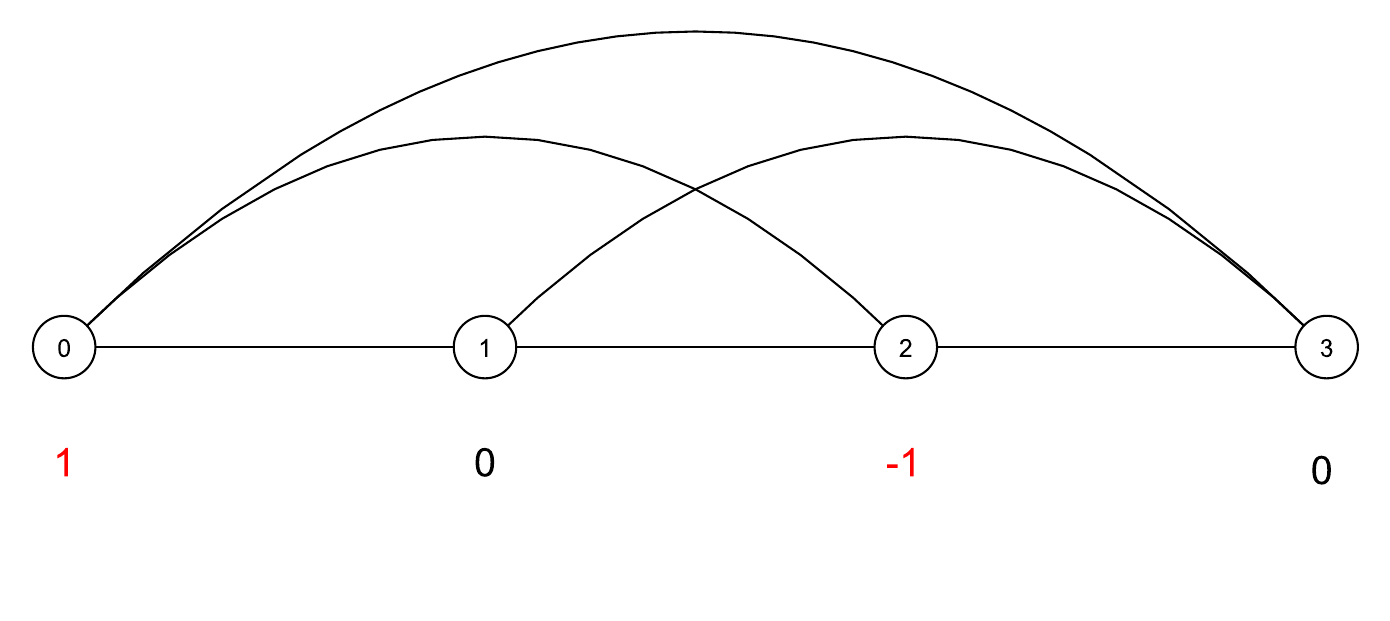}
\end{center}
\caption{Function $a_2(0,2)$ for $q=4$.}\label{A2}
\end{figure}

\begin{figure}[H]
 \begin{center}
\includegraphics[scale=0.35]{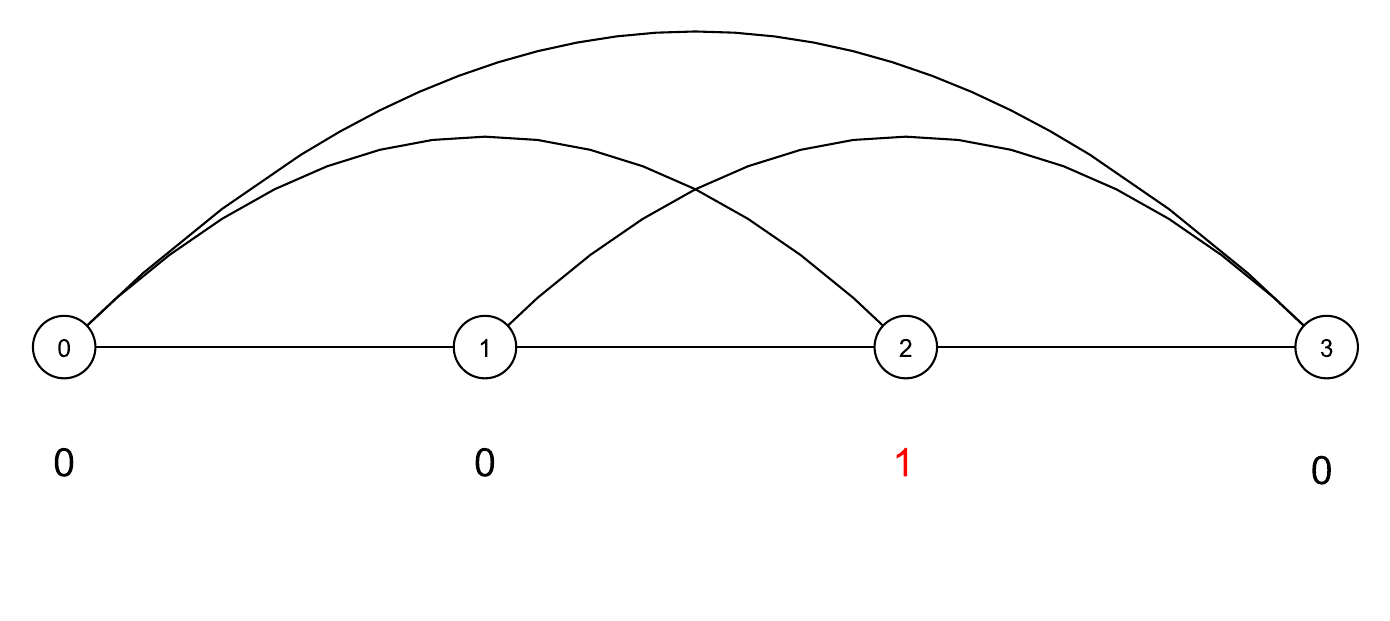}
\end{center}
\caption{Function $a_4(2)$ for $q=4$.}\label{A4}
\end{figure}

The following lemma is a particular case of well known result for so--called NEPS of graphs (see \cite{Cvetkovic}, Theorem 2.3.4):
\begin{lemma}\label{eigenvectors of product}
Let $G_1$ and $G_2$ be two graphs, let $\lambda$ and $\mu$ be eigenvalues of $G_1$ and $G_2$ respectively, and let $x$ and $y$ be eigenvectors for $\lambda$ and $\mu$. Then the graph $G_1\square G_2$ has the eigenvalue $\lambda+\mu$ and $x\otimes y$ is an eigenvector corresponding to $\lambda+\mu$.

\end{lemma}

Since $$\lambda_{i}(m,q)+\lambda_{j}(n,q)=\lambda_{i+j}(m+n,q)$$ and $H(m+n,q)=H(m,q)\square H(n,q)$, by Lemma~\ref{eigenvectors of product} we immediately obtain the following corollary:
\begin{corollary}\label{product of functions}
Let $f_1\in{U_{i}(m,q)}$ and $f_2\in{U_{j}(n,q)}$. Then $f_1\cdot f_2\in{U_{i+j}(m+n,q)}$.

\end{corollary}

Let $i+j\le n$. We say that a function $f:\Sigma_{q}^n\longrightarrow{\mathbb{R}}$ belongs to the class $F_{1}(n,q,i,j)$ if $$f=c\cdot \prod_{k=1}^{i}g_{k}\cdot \prod_{k=1}^{n-i-j}h_{k}\cdot \prod_{k=1}^{j-i}v_{k},$$ where $c$ is a constant, $g_k\in{A_1}$ for $k\in{[1,i]}$, $h_k\in{A_3}$ for $k\in{[1,n-i-j]}$ and $v_k\in{A_4}$ for $k\in{[1,j-i]}$.

Let $i+j>n$. We say that a function $f:\Sigma_{q}^n\longrightarrow{\mathbb{R}}$ belongs to the class $F_{2}(n,q,i,j)$ if $$f=c\cdot \prod_{k=1}^{n-j}g_{k}\cdot \prod_{k=1}^{i+j-n}h_{k}\cdot \prod_{k=1}^{j-i}v_{k},$$ where $c$ is a constant, $g_k\in{A_1}$ for $k\in{[1,n-j]}$, $h_k\in{A_2}$ for $k\in{[1,i+j-n]}$ and $v_k\in{A_4}$ for $k\in{[1,j-i]}$.
\begin{lemma}\label{constructions of eigenfunctions}
The following statements are true:
\begin{enumerate}
\item Let $i+j\le n$ and $f\in{F_{1}(n,q,i,j)}$. Then $f\in{U_{[i,j]}(n,q)}$ and $|f|=2^{i}(q-1)^{i}q^{n-i-j}$.

  \item Let $i+j>n$ and $f\in{F_{2}(n,q,i,j)}$. Then $f\in{U_{[i,j]}(n,q)}$ and $|f|=2^{i}(q-1)^{n-j}$.

\end{enumerate}

\end{lemma}
\begin{proof}
As we noted above $A_1\subset U_{1}(2,q)$, $A_2\subset U_{1}(1,q)$, $A_3\subset U_{0}(1,q)$ and $A_4\subset U_{[0,1]}(1,q)$. Then using Corollary~\ref{product of functions} and the fact that $|f_1\cdot f_2|=|f_2|\cdot |f_2|$, we obtain the statement of this lemma.
\end{proof}

In what follows, we prove that functions from $F_{1}(n,q,i,j)$ and $F_{2}(n,q,i,j)$ have the minimum size of the support in the subspace $U_{[i,j]}(n,q)$ for $i+j\le n$, $q\geq3$ and for $i+j>n$, $q\geq4$ respectively.
\section{Reduction Lemma}
In this section we describe a connection between eigenspaces of the Hamming graphs $H(n,q)$ and $H(n-1,q)$.

\begin{lemma}\label{reduction 1}

Let  $f\in{U_{i}(n,q)}$ and $r\in\{1,2,\ldots,n\}$. Then the following statements are true:
\begin{enumerate}
\item $f_{k}^{r}-f_{m}^{r}\in{U_{i-1}(n-1,q)}$ for $k,m\in\Sigma_q$.

\item $\sum_{k=0}^{q-1}f_{k}^{r}\in{U_{i}(n-1,q)}$.
\item $f_{k}^{r}\in{U_{i-1}(n-1,q)\oplus U_{i}(n-1,q)}$ for $k\in\Sigma_q$.
\end{enumerate}
\end{lemma}
\begin{proof}

1. The first case of this lemma was proved in \cite{Valyuzhenich} (Lemma 1).

2. Let $t=(t_1,t_2,\ldots,t_n)$ be a vertex of $H(n,q)$. Let $$x_{r}(m)=(t_1,\ldots,t_{r-1},m,t_{r+1},\ldots,t_n)$$ for $m\in\Sigma_q$ and $$x_{i,r}(a,m)=(t_1,\ldots,t_{i-1},a,t_{i+1},\ldots,t_{r-1},m,t_{r+1},\ldots,t_n)$$ for $a,m\in\Sigma_q$ and $i\in\{1,2,\ldots,n\}\setminus \{r\}$.  The set of neighbors $z=(z_1,\ldots,z_n)$ of $x_{r}(m)$ such that $z_r=m$ is denoted by $N(m,r)$. We see that $N(m,r)=\{x_{i,r}(a,m)\mid i\neq{r},a\neq t_i\}$. We note that $$N(x_{r}(m))=(\{x_{r}(0),x_{r}(1),\ldots,x_{r}(q-1)\}\setminus\{x_{r}(m)\})\cup\ N(m,r).$$ Since $f$ is an eigenfunction, we have
$$\lambda_{i}(n,q)\cdot f(x_{r}(m))=\sum_{i\neq{r},a\neq{t_i}}f(x_{i,r}(a,m))+\sum_{i=0}^{q-1}f(x_{r}(i))-f(x_{r}(m)).$$
Hence we obtain that $$(\lambda_{i}(n,q)-(q-1))\cdot \sum_{m=0}^{q-1}f(x_{r}(m))=\sum_{i\neq{r},a\neq{t_i}}\sum_{m=0}^{q-1}f(x_{i,r}(a,m)).$$

Let $y_r$ and $y_{i}(a)$ be the vectors obtained by removing the $r$th coordinate in $x_{r}(m)$ and $x_{i,r}(a,m)$ respectively.
Then $$\lambda_{i}(n-1,q)\cdot (\sum_{m=0}^{q-1}f_{m}^{r})(y_r)=\sum_{i\neq{r},a\neq{t_i}}(\sum_{m=0}^{q-1}f_{m}^{r})(y_{i}(a)).$$
Since $y_r$ has neighbors $y_i(a)$ for $i\neq r$ and $a\neq t_i$ in $H(n-1,q)$, we prove that $\sum_{m=0}^{q-1}f_{m}^{r}$ is a $\lambda_{i}(n-1,q)$-eigenfunction of $H(n-1,q)$.

3. By the first case of this lemma we have that $f_{k}^{r}-f_{m}^{r}\in{U_{i-1}(n-1,q)}$ for $m\neq k$. Hence $$(q-1)f_{k}^{r}-\sum_{t\in\Sigma_q,t\neq k}f_{t}^{r}\in{U_{i-1}(n-1,q)}.$$
The second case of this lemma implies that $$\sum_{t=0}^{q-1}f_{t}^{r}\in{U_{i}(n-1,q)}.$$
Hence $q\cdot f_{k}^{r}\in{U_{i-1}(n-1,q)\oplus U_{i}(n-1,q)}$. Thus $f_{k}^{r}\in{U_{i-1}(n-1,q)\oplus U_{i}(n-1,q)}$.
\end{proof}

Using the previous lemma for $U_{k}(n,q)$, where $i\le k\le j$, we obtain the following result:
\begin{lemma}\label{reduction 2}

Let $f\in{U_{[i,j]}(n,q)}$ and $r\in\{1,2,\ldots,n\}$. Then the following statements are true:
\begin{enumerate}
\item $f_{k}^{r}-f_{m}^{r}\in{U_{[i-1,j-1]}(n-1,q)}$ for $k,m\in\Sigma_q$.

\item $\sum_{k=0}^{q-1}f_{k}^{r}\in{U_{[i,j]}(n-1,q)}$.
\item $f_{k}^{r}\in{U_{[i-1,j]}(n-1,q)}$ for $k\in\Sigma_q$.
\end{enumerate}
\end{lemma}

\begin{lemma}\label{reduction 3}
Let $f\in{U_{[i,j]}(n,q)}$, let $r\in\{1,2,\ldots,n\}$, and let $m\in\Sigma_q$. If $f_{k}^{r}\equiv 0$ for any $k\in\Sigma_{q}\setminus\{m\}$, then $f_{m}^{r}\in{U_{[i,j-1]}(n,q)}$.
\end{lemma}
\begin{proof}
For $k\neq m$ we obtain $f_{k}^{r}-f_{m}^{r}\in{U_{[i-1,j-1]}(n-1,q)}$ due to Lemma~\ref{reduction 2}(1). Hence $f_{m}^{r}\in{U_{[i-1,j-1]}(n-1,q)}$. Lemma~\ref{reduction 2}(2) implies that $f_{m}^r\in{U_{[i,j]}(n-1,q)}$. Then $f_{m}^{r}\in{U_{[i,j-1]}(n-1,q)}$.
\end{proof}

\begin{lemma}\label{nizhnya ocenka for nositelya}
Let $f:\Sigma_{q}^{n}\longrightarrow{\mathbb{R}}$, let $r\in\{1,2,\ldots,n\}$ and let $f_{0}^{r}=f_{1}^{r}=\ldots=f_{q-2}^{r}$. Then $$|f|\geq(q-2)|f_{0}^{r}|+|f_{q-2}^{r}-f_{q-1}^{r}|.$$
\end{lemma}
\begin{proof}
We have $$|f|=\sum_{k=0}^{q-1}|f_{k}^{r}|=(q-2)|f_{0}^{r}|+|f_{q-2}^{r}|+|f_{q-1}^{r}|\geq{(q-2)|f_{0}^{r}|+|f_{q-2}^{r}-f_{q-1}^{r}|}.$$
\end{proof}

In Sections 5 and 6 we will use the main results of this section for inductive arguments.

\section{Case $i+j\le n$}
In this section we prove the first main theorem of this paper:
\begin{theorem}\label{Theorem 1.1}
Let $f\in{U_{[i,j]}(n,q)}$, $i+j\le n$, $q\geq 3$ and $f\not\equiv 0$. Then $|f|\geq 2^{i}(q-1)^{i}q^{n-i-j}$. Moreover, the equality $|f|=2^{i}(q-1)^{i}q^{n-i-j}$ holds if and only if $f_{\sigma}\in F_{1}(n,q,i,j)$ for some permutation $\sigma\in{S_n}$.
\end{theorem}
\begin{proof}
Lemma~\ref{constructions of eigenfunctions} implies that if $f\in F_{1}(n,q,i,j)$, then $f\in{U_{[i,j]}(n,q)}$ and $|f|=2^{i}(q-1)^{i}q^{n-i-j}$.

In what follows, in this theorem we assume that $|f|\leq 2^{i}(q-1)^{i}q^{n-i-j}$.
Let us prove the theorem by induction on $n$, $i$ and $j$.
Suppose that $f$ is a constant. Then $f$ is a $\lambda_{0}(n,q)$-eigenfunction of $H(n,q)$, i.e $f\in U_{0}(n,q)$. In this case $|f|=q^n$ and the claim of the theorem holds.
So, we can assume that $f_{k}^{r}\neq f_{m}^{r}$ for some $k,m\in\Sigma_q$ and $r\in\{1,2,\ldots,n\}$. Without loss of generality, we assume that $r=n$. For the function $f_{k}^{n}$ in the proof of this theorem we will use the more convenient notation $f_k$.

Now we prove the theorem for $i=0$.
\begin{lemma}\label{Lemma 1.1}
Let $f\in{U_{[0,j]}(n,q)}$, $j\le n$, $q\geq 3$ and $f\not\equiv 0$. Then $|f|\geq{q^{n-j}}$. Moreover $|f|={q^{n-j}}$ if and only if $f_{\sigma}\in F_{1}(n,q,0,j)$ for some permutation $\sigma\in{S_n}$.

\end{lemma}
\begin{proof}
We assume that $|f|\leq{q^{n-j}}$.
Let us prove this lemma by induction on $n$ and $j$. For $j=0$, we have $f$ is a constant. So $|f|={q^{n}}$ and the claim of the lemma holds. If $n=1$ and $j>0$, then $j=1$. In this case the claim of the theorem also holds. So, in this lemma we can assume that $n\geq 2$. Let us prove the induction step. As we noted above there exist numbers $k$ and $m$ such that $f_k\neq f_m$ and $f_m\not\equiv0$. Without loss of generality, we can assume that $k=q-2$ and $m=q-1$.

Lemma~\ref{reduction 2}(1) implies that $f_{q-2}-f_{q-1}\in{U_{[0,j-1]}(n-1,q)}$.
By the induction assumption, $|f_{q-2} - f_{q-1}| \ge q^{n-j}$. So,
$$ |f| \ge |f_{q-2}| + |f_{q-1}| \ge |f_{q-2} - f_{q-1}| \ge q^{n-j}.$$
On the other hand, we supposed that $|f| \le q^{n-j}$.
So, we have $f_k \equiv 0$ for every $k<q-2$. In particular,
$f_0 \equiv 0$ because $q \ge 3$.
In particular, $f_0 \ne f_{q-1}$, and considering $k=0$ and $m=q-1$
similarly as above,
we find that $f_{q-2} \equiv 0$ too.
Thus,
$$ f(x_1,\ldots,x_n)=f_{q-1}(x_1,\ldots,x_{n-1})\cdot a_4(q-1)(x_n), $$
and the statement of the lemma follows from the induction assumption applied to $f_{q-1}$.

\end{proof}

Further we will prove the theorem for $i\geq 1$. We note that if $n\leq 2$, $i+j\le n$ and $i\geq 1$, then $n=2$ and $i=j=1$. In this case the statement of Theorem~\ref{Theorem 1.1} was proved in \cite{Valyuzhenich} (Theorem 3). In what follows, in the proof of the theorem we assume that $n\geq 3$.

\begin{lemma}\label{Lemma o ravnomernosti 1}
Let $f$ be a non uniform function from $U_{[i,j]}(n,q)$, where $i+j\le n$, $i\geq 1$ and $q\geq 3$. Then $|f|>2^{i}(q-1)^{i}q^{n-i-j}$.
\end{lemma}
\begin{proof}

\textbf{Case $q>3$.} Since $f$ is not a uniform function, there exist a number $r$ and distinct numbers $k$, $m$, $s$ and $t$ such that $f_{k}^{r}\neq{f_{m}^{r}}$ and $f_{s}^{r}\neq{f_{t}^{r}}$. Denote $\tilde{f_k}=f_{k}^{r}$ for $k\in\Sigma_q$. Lemma~\ref{reduction 2}(1) implies that $\tilde{f_{k}}-\tilde{f_{m}}\in{U_{[i-1,j-1]}(n-1,q)}$ and  $\tilde{f_{s}}-\tilde{f_{t}}\in{U_{[i-1,j-1]}(n-1,q)}$. By the induction assumption we have $$|\tilde{f_{k}}-\tilde{f_{m}}|\geq 2^{i-1}(q-1)^{i-1}q^{n-i-j+1}$$ and $$|\tilde{f_{s}}-\tilde{f_{t}}|\geq 2^{i-1}(q-1)^{i-1}q^{n-i-j+1}.$$ Then $$|f|=\sum_{p=0}^{q-1}|\tilde{f_p}|\geq |\tilde{f_k}|+|\tilde{f_m}|+|\tilde{f_s}|+|\tilde{f_t}|\geq |\tilde{f_{k}}-\tilde{f_{m}}|+ |\tilde{f_{s}}-\tilde{f_{t}}| \geq{2^{i}(q-1)^{i-1}q^{n-i-j+1}}>2^{i}(q-1)^{i}q^{n-i-j}.$$

\textbf{Case $q=3$.} Since $f$ is not a uniform function, there exists a number $r$ such that $f_{0}^{r}\neq f_{1}^{r}$, $f_{1}^{r}\neq f_{2}^{r}$ and $f_{0}^{r}\neq f_{2}^{r}$. Denote $\tilde{f_k}=f_{k}^{r}$ for $k\in\Sigma_q$. Lemma~\ref{reduction 2}(1) implies that $\tilde{f_{k}}-\tilde{f_{m}}\in{U_{[i-1,j-1]}(n-1,q)}$ for $k\neq m$. Then by the induction assumption we obtain that $$|\tilde{f_{k}}-\tilde{f_{m}}|\geq 2^{2i-2}\cdot3^{n-i-j+1}$$  for $k,m\in{\Sigma_3}$ and $k\neq m$. We note that $$|\tilde{f_0}|+|\tilde{f_1}|+|\tilde{f_2}|=\frac{1}{2}(|\tilde{f_0}|+|\tilde{f_1}|)+\frac{1}{2}(|\tilde{f_1}|+|\tilde{f_2}|)+\frac{1}{2}(|\tilde{f_0}|+|\tilde{f_2}|).$$
Using inequalities $|\tilde{f_{k}}|+|\tilde{f_{m}}|\geq |\tilde{f_{k}}-\tilde{f_{m}}|$ for $k,m\in{\Sigma_3}$ and $k\neq m$, we obtain that $$|\tilde{f_0}|+|\tilde{f_1}|+|\tilde{f_2}|\geq{2^{2i-3}\cdot3^{n-i-j+2}}.$$ Then $$|f|\geq{2^{2i-3}\cdot3^{n-i-j+2}}>2^{2i}\cdot3^{n-i-j}.$$

\end{proof}
\begin{lemma}\label{Lemma o tom cto mozno induction delat 1}
Let $f\in{U_{[i,j]}(n,q)}$, $r\in\{1,2,\ldots,n\}$, $f_{0}^r=f_{1}^r=\ldots=f_{q-2}^r$, $f_{0}^r\neq f_{q-1}^r$, $i+j\le n$, $i\geq 1$  and $q\geq 3$. Let $|f_{0}^r|>2^{i-1}(q-1)^{i-1}q^{n-i-j}$. Then $|f|>2^{i}(q-1)^{i}q^{n-i-j}$.
\end{lemma}
\begin{proof}
 Lemma~\ref{reduction 2}(1) implies that $f_{q-2}^r-f_{q-1}^r\in{U_{[i-1,j-1]}(n-1,q)}$. Hence by the induction assumption we obtain that $$|f_{q-2}^r-f_{q-1}^r|\geq 2^{i-1}(q-1)^{i-1}q^{n-i-j+1}.$$ By Lemma~\ref{nizhnya ocenka for nositelya} we have $$|f|\geq{(q-2)|f_{0}^r|+|f_{q-2}^r-f_{q-1}^r|}.$$ Hence $$|f|>{(q-2)2^{i-1}(q-1)^{i-1}q^{n-i-j}+2^{i-1}(q-1)^{i-1}q^{n-i-j+1}}=2^{i}(q-1)^{i}q^{n-i-j}.$$
\end{proof}
We continue the proof of Theorem \ref{Theorem 1.1}.
Recall that we assume $|f|\leq 2^{i}(q-1)^{i}q^{n-i-j}$. Then using Lemma~\ref{Lemma o ravnomernosti 1}, we obtain that $f$ is a uniform function. Hence, without loss of generality, we can assume that $f_0=f_1=\ldots=f_{q-2}$. Lemma~\ref{Lemma o tom cto mozno induction delat 1} implies that
$|f_0|\leq 2^{i-1}(q-1)^{i-1}q^{n-i-j}$. We have $f_{0}\in{U_{[i-1,j]}(n-1,q)}$ due to Lemma~\ref{reduction 2}(3). Then by the induction assumption there are two cases: $f_0\equiv0$ or $|f_0|=2^{i-1}(q-1)^{i-1}q^{n-i-j}$.

Consider the case $f_0\equiv0$. Lemma~\ref{reduction 3} implies that $f_{q-1}\in{U_{[i,j-1]}(n-1,q)}$. If $i=j$, then $f_{q-1}\in{U_{[i,i-1]}(n-1,q)}$
and we have that $f_{q-1}\equiv0$. Hence $f\equiv0$ for $i=j$. So, we can assume that $i<j$. By the induction assumption we obtain $|f_{q-1}|\geq2^{i}(q-1)^{i}q^{n-i-j}$. Then  $$|f|=\sum_{k=0}^{q-1}|f_k|=|f_{q-1}|\geq2^{i}(q-1)^{i}q^{n-i-j}.$$ Moreover, if $|f|=2^{i}(q-1)^{i}q^{n-i-j}$, then $|f_{q-1}|=2^{i}(q-1)^{i}q^{n-i-j}$. Then by the induction assumption for $f_{q-1}$ we obtain that  $$(f_{q-1})_{\pi}\in{F_1(n-1,q,i,j-1)}$$ for some permutation $\pi\in{S_{n-1}}$. Since $f_0=f_1=\ldots=f_{q-2}\equiv{0}$, we have $f=f_{q-1}\cdot a_{4}(q-1)$ and $a_{4}(q-1)\in{A_4}$. So, we prove the theorem in this case.

Consider the case $|f_{0}|=2^{i-1}(q-1)^{i-1}q^{n-i-j}$.
Lemma \ref{reduction 2}(3) implies that $f_{0}\in{U_{[i-1,j]}(n-1,q)}$. Then by the induction assumption for $f_{0}$ we obtain that
 $$(f_{0})_{\pi}\in{F_1(n-1,q,i-1,j)}$$ for some permutation $\pi\in{S_{n-1}}$. Without loss of generality one can take $\pi$ equal the identity permutation, so $f_{0}\in{F_1(n-1,q,i-1,j)}$. Hence $f_0=g\cdot a_4(m)$ for some $m\in{\Sigma_{q}}$. Without loss of generality, we assume that $m=q-1$.
Therefore, we have $f_k=g\cdot a_4(q-1)$ for any $k<q-1$. Then $f|_{x_{n-1}=a,x_{n}=b}\equiv0$ for $a\in{[0,q-2]}$ and $b\in{[0,q-2]}$ and $f|_{x_{n-1}=q-1,x_{n}=c}=g$ for any $c<q-1$. We also note that $$|g|=|f_0|=2^{i-1}(q-1)^{i-1}q^{n-i-j}.$$

Let us consider the functions $f_{0}^{n-1},f_{1}^{n-1},\ldots,f_{q-1}^{n-1}$.
Since $f_{k}^{n-1}|_{x_n=0}\equiv0$ for any $k<q-1$ and $f_{q-1}^{n-1}|_{x_n=0}=g$, we see that $f_{k}^{n-1}\neq f_{q-1}^{n-1}$ for any $k<q-1$.
On the other hand, $f$ is uniform. Hence $f_{0}^{n-1}=\ldots=f_{q-2}^{n-1}$.
If $f_{0}^{n-1}\equiv0$, then we have the case that we considered above (we can consider $f_{0}^{n-1}$ instead of $f_0$). So $f_{0}^{n-1}\not\equiv0$.
By Lemma \ref{reduction 2}(3) we have $f_{0}^{n-1}\in{U_{[i-1,j]}(n-1,q)}$. Then by the induction assumption we obtain
$$|f_{0}^{n-1}|\geq 2^{i-1}(q-1)^{i-1}q^{n-i-j}.$$
Denote $h=f_{0}^{n-1}|_{x_{n}=q-1}$. Then $$|h|=|f_{0}^{n-1}|\geq2^{i-1}(q-1)^{i-1}q^{n-i-j}.$$
We also note that $f_{k}^{n-1}|_{x_{n}=q-1}=h$ for any $k<q-1$. Denote $\varphi=f|_{x_{n-1}=q-1,x_{n}=q-1}$.

Recall that $|f|\leq 2^{i}(q-1)^{i}q^{n-i-j}$. On the other hand, we have $$|f|=(q-1)(|g|+|h|)+|\varphi|,$$
$|g|=2^{i-1}(q-1)^{i-1}q^{n-i-j}$ and $|h|\geq2^{i-1}(q-1)^{i-1}q^{n-i-j}$.
So $|h|=2^{i-1}(q-1)^{i-1}q^{n-i-j}$ and $\varphi\equiv0$.

Let us prove that $g+h\in{U_{[i-1,j-2]}(n-2,q)}$.
Since $f_{0}\in{U_{[i-1,j]}(n-1,q)}$, by Lemma \ref{reduction 3} we have $g\in{U_{[i-1,j-1]}(n-2,q)}$. Similarly we obtain that $h\in{U_{[i-1,j-1]}(n-2,q)}$. Consequently, we have $g+h\in{U_{[i-1,j-1]}(n-2,q)}$. On the other hand, Lemma \ref{reduction 2}(1) implies that $f_{q-2}-f_{q-1}\in{U_{[i-1,j-1]}(n-1,q)}$. Applying Lemma~\ref{reduction 2}(1) for $f_{q-1}-f_{q-2}$, we obtain that $g+h\in{U_{[i-2,j-2]}(n-2,q)}$. Therefore $g+h\in{U_{[i-1,j-2]}(n-2,q)}$.

Let us prove that $g+h\equiv0$. Suppose $g+h\not\equiv0$. Then by the induction assumption we have
$$|g+h|\geq2^{i-1}(q-1)^{i-1}q^{n-i-j+1}.$$
On the other hand, $|g|=|h|=2^{i-1}(q-1)^{i-1}q^{n-i-j}$.
So, we have $$2^{i}(q-1)^{i-1}q^{n-i-j}=|g|+|h|\geq|g+h|\geq2^{i-1}(q-1)^{i-1}q^{n-i-j+1}.$$
Therefore $q\le 2$ and we have a contradiction. Thus $g+h\equiv0$. Then $f=g\cdot a_{1}(q-1,q-1)$. Applying the induction assumption for $g$, we finish the proof of the theorem.

\end{proof}

\begin{corollary}\label{Corollary 1.1}
Let $f\in{U_{i}(n,q)}$, $i\leq \lfloor\frac{n}{2}\rfloor$, $q\geq 3$ and $f\not\equiv 0$. Then $|f|\geq 2^{i}(q-1)^{i}q^{n-2i}$. Moreover, the equality $|f|=2^{i}(q-1)^{i}q^{n-2i}$ holds if and only if $f_{\sigma}\in F_{1}(n,q,i,i)$ for some permutation $\sigma\in{S_n}$.
\end{corollary}

\section{Case $i+j>n$}

In this section we prove the second main result of this work. We find the minimum size of the support of functions from $U_{[i,j]}(n,q)$ for $i+j>n$.
Firstly, we solve the problem for the uniform functions:

\begin{theorem}\label{Theorem for uniform}
Let $f$ be a uniform function from $U_{[i,j]}(n,q)$, where $i+j\geq n$, $q\geq 3$ and $f\not\equiv 0$. Then $|f|\geq{2^{n-j}(q-1)^{n-j}q^{i+j-n}}$.
\end{theorem}
\begin{proof}
We note that the statement of the theorem was proved in Theorem~\ref{Theorem 1.1} for $n=i+j$. So, we can assume that $i+j>n$.
Let us prove this theorem by induction on $n$, $i$ and $j$.
Consider the functions $f_{k}^{n}$ for $k\in\Sigma_q$. For the function $f_{k}^{n}$ we will use the more convenient notation $f_k$.
Since $f$ is a uniform function, we can assume that $f_0=f_1=\ldots=f_{q-2}$. We note that $f_k$ is uniform for any $k\in\Sigma_q$.

Firstly, we prove the theorem for $j=n$.
\begin{lemma}\label{Lemma Case n=j}
Let $f$ be a uniform function from $U_{[i,n]}(n,q)$, where $q\geq 3$ and $f\not\equiv 0$. Then $|f|\geq{q^{i}}$.
\end{lemma}
\begin{proof}
Let us prove this lemma by induction on $n$ and $i$.
For $i=0$ and arbitrary $n$, we see that $|f|\geq{1}$.

Let us prove the induction step.
 Suppose $f_0\equiv 0$. Lemma~\ref{reduction 3} implies that $f_{q-1}\in{U_{[i,n-1]}(n-1,q)}$. Then by the induction assumption we obtain $|f_{q-1}|\geq q^{i}$. Hence $|f|=|f_{q-1}|\geq q^{i}$.

Consider the case $f_{q-1}\equiv 0$. Using Lemma~\ref{reduction 2}(2), we have $f_{0}\in{U_{[i,n-1]}(n-1,q)}$. Then by the induction assumption we obtain $|f_{0}|\geq q^{i}$. Therefore $|f|>{q^{i}}$.

Suppose $f_0\not\equiv 0$ and $f_{q-1}\not\equiv 0$. We have $f_{k}\in{U_{[i-1,n-1]}(n-1,q)}$ for $k\in{\Sigma_q}$ due to Lemma~\ref{reduction 2}(3). By the induction assumption we have $|f_k|\geq q^{i-1}$ for $k\in{\Sigma_q}$. Hence $|f|\geq q^{i}$.

\end{proof}

Now we prove the theorem for $j<n$.
We consider two cases.

\textbf{Case $f_0\equiv 0$.} Since $f\not\equiv 0$, we have $f_{q-1}\not\equiv 0$.  Lemma~\ref{reduction 3} implies that $f_{q-1}\in{U_{[i,j-1]}(n-1,q)}$. Then by the induction assumption we obtain $|f_{q-1}|\geq 2^{n-j}(q-1)^{n-j}q^{i+j-n}$. Hence $|f|=|f_{q-1}|\geq 2^{n-j}(q-1)^{n-j}q^{i+j-n}$.

\textbf{Case $f_0\not\equiv 0$.}
Lemma~\ref{reduction 2}(2) implies that $$\sum_{p=0}^{q-1}f_{p}=(q-1)\cdot f_0+f_{q-1}\in{U_{[i,j]}(n-1,q)}.$$ If $(q-1)f_0+f_{q-1}\equiv0$, then $f_0-f_{q-1}=q\cdot f_0$. Then $f_{k}\in{U_{[i-1,j-1]}(n-1,q)}$ for $k\in{\Sigma_q}$ due to Lemma~\ref{reduction 2}(1). Recall that in the beginning of the proof we assumed $i+j>n$. By the induction assumption we have $$|f_k|\geq 2^{n-j}(q-1)^{n-j}q^{i+j-n-1}.$$ Hence $|f|\geq 2^{n-j}(q-1)^{n-j}q^{i+j-n}$.

Suppose $(q-1)f_0+f_{q-1}\not\equiv0$. We note that $(q-1)f_0+f_{q-1}$ is uniform. By the induction assumption we have $$|(q-1)f_0+f_{q-1}|\geq 2^{n-j-1}(q-1)^{n-j-1}q^{i+j-n+1}.$$
Since $|f_{0}|+|f_{q-1}|\geq |(q-1)\cdot f_{0}+f_{q-1}|$, we obtain $$|f_{0}|+|f_{q-1}|\geq 2^{n-j-1}(q-1)^{n-j-1}q^{i+j-n+1}.$$
By Lemma~\ref{reduction 2}(3), we have $f_{0}\in{U_{[i-1,j]}(n-1,q)}$, and the induction assumption implies that $$|f_{0}|\geq 2^{n-j-1}(q-1)^{n-j-1}q^{i+j-n}.$$
Using the equality $$|f|=\sum_{p=0}^{q-1}|f_{p}|=(q-2)\cdot |f_{0}|+|f_0|+|f_{q-1}|,$$ we obtain $|f|\geq 2^{n-j}(q-1)^{n-j}q^{i+j-n}$.

\end{proof}

Now we prove the main theorem of this section.
\begin{theorem}\label{Theorem 2.1}
Let $f\in{U_{[i,j]}(n,q)}$, $i+j>n$, $q\geq 4$ and $f\not\equiv 0$. Then $|f|\geq 2^{i}(q-1)^{n-j}$. Moreover, for $i=j$ and $q\geq 5$ the equality $|f|=2^{i}(q-1)^{n-i}$ holds if and only if $f_{\sigma}\in F_{2}(n,q,i,i)$ for some permutation $\sigma\in{S_{n}}$.
\end{theorem}
\begin{proof}

Lemma~\ref{constructions of eigenfunctions} implies that if $f\in F_{2}(n,q,i,j)$, then $f\in{U_{[i,j]}(n,q)}$ and $|f|=2^{i}(q-1)^{n-j}$.

Let us prove this theorem by induction on $n$, $i$ and $j$. Since $i+j>n$, we have that $i\geq 1$.
Suppose that there exist numbers $k$ and $r$ such that $f_{k}^{r}\equiv 0$. Without loss of generality, we assume that $k=q-1$ and $r=n$. For the function $f_{k}^{n}$ we will use the more convenient notation $f_k$.
Lemma~\ref{reduction 2}(1) implies that $f_{m}-f_{q-1}\in{U_{[i-1,j-1]}(n-1,q)}$ for $m<q-1$. Therefore $f_{m}\in{U_{[i-1,j-1]}(n-1,q)}$ for $m<q-1$. So, if $f_{m}\not\equiv0$, then using the induction assumption for $i+j>n+1$ and Theorem~\ref{Theorem 1.1} in the case $i+j=n+1$, we have $|f_m|\geq 2^{i-1}(q-1)^{n-j}$. Since $|f|=\sum_{p=0}^{q-1}|f_p|$, the number of $k$ such that $f_{k}\not\equiv0$ is at most two. There are two variants.

In the first case there exists only one $k$ such that $f_{k}\not\equiv0$. Without loss of generality, we assume that $k=0$. We have $f_{0}\in{U_{[i,j-1]}(n-1,q)}$ due to Lemma~\ref{reduction 3}. If $i=j$, then $f_{0}\in{U_{[i,i-1]}(n-1,q)}$ and $f\equiv0$. For $i<j$ by the induction assumption we obtain $|f_0|\geq 2^{i}(q-1)^{n-j}$. So $|f|\geq 2^{i}(q-1)^{n-j}$.

In the second case there exist two numbers $k$ and $m$ such that $f_{k}\not\equiv0$ and $f_{m}\not\equiv0$. Without loss of generality, we assume that $k=0$ and $m=1$. As we noted above $|f_0|\geq 2^{i-1}(q-1)^{n-j}$ and $|f_1|\geq 2^{i-1}(q-1)^{n-j}$. So
$|f|=|f_0|+|f_1|\geq 2^{i}(q-1)^{n-j}$. Suppose that $i=j$, $q\geq 5$ and the equality $|f|=2^{i}(q-1)^{n-i}$ holds. By Lemma~\ref{reduction 2}(2) we obtain that $f_{0}+f_{1}\in{U_{i}(n-1,q)}$. Since $f_{0}\in{U_{i-1}(n-1,q)}$ and $f_{1}\in{U_{i-1}(n-1,q)}$, we see that $f_{0}+f_{1}\in{U_{i-1}(n-1,q)}$. Consequently $f_{0}+f_{1}\equiv0$. Hence $f=f_0\cdot a_2(0,1)$. Since $|f|=2^{i}(q-1)^{n-i}$, we have $|f_0|=2^{i-1}(q-1)^{n-i}$. Applying the induction assumption for $f_0$ we prove this theorem.

Thus, in what follows in the proof of this theorem we can assume that $f_{k}^{v}\not\equiv0$ for any $k\in{\Sigma_q}$ and $v\in\{1,2,\ldots,n\}$.

We need the following lemma.
\begin{lemma}\label{Lemma o ravnomernosti 2}
Let $f$ be a non uniform function from $U_{[i,j]}(n,q)$, where $i+j>n$, $f_{k}^{v}\not\equiv0$ for $k\in{\Sigma_q}$ and $v\in\{1,2,\ldots,n\}$, $i\geq 1$ and $q\geq 4$. Then $|f|>2^{i}(q-1)^{n-j}$ for $q>4$ and $|f|\geq2^{i}(q-1)^{n-j}$ for $q=4$.
\end{lemma}
\begin{proof}

Since $f$ is not a uniform function, there exist number $r$ and distinct numbers $k$, $m$, $s$ and $t$ such that $f_{k}^r\neq{f_{m}^{r}}$ and $f_{s}^{r}\neq{f_{t}^{r}}$. Denote $\tilde{f_k}=f_{k}^{r}$ for $k\in\Sigma_q$. Lemma~\ref{reduction 2}(1) implies that $\tilde{f_{k}}-\tilde{f_{m}}\in{U_{[i-1,j-1]}(n-1,q)}$ and  $\tilde{f_{s}}-\tilde{f_{t}}\in{U_{[i-1,j-1]}(n-1,q)}$. Then using the induction assumption for $i+j>n+1$ and Theorem~\ref{Theorem 1.1} in the case $i+j=n+1$, we obtain that $$|\tilde{f_{k}}-\tilde{f_{m}}|\geq 2^{i-1}(q-1)^{n-j}$$ and $$|\tilde{f_{s}}-\tilde{f_{t}}|\geq 2^{i-1}(q-1)^{n-j}.$$

Therefore, we have $$|f|=\sum_{p=0}^{q-1}|\tilde{f_p}|\geq{|\tilde{f_k}|+|\tilde{f_m}|+|\tilde{f_s}|+|\tilde{f_t}|}\geq |\tilde{f_{k}}-\tilde{f_{m}}|+|\tilde{f_{s}}-\tilde{f_{t}}|\geq 2^{i}(q-1)^{n-j}.$$

If $q>4$, then there exists $d$ such that $d\not\in\{k,m,s,t\}$ and $\tilde{f_{d}}\not\equiv0$. So $|f|>2^{i}(q-1)^{n-j}$ for $q>4$.
\end{proof}

Now we finish the proof of this theorem. Suppose that $f$ is not a uniform function. Since $f_{k}^{v}\not\equiv0$ for any $k$ and $v$, by  Lemma~\ref{Lemma o ravnomernosti 2} we obtain that $|f|>2^{i}(q-1)^{n-j}$ for $q>4$ and $|f|\geq2^{i}(q-1)^{n-j}$ for $q=4$. So, we can assume that $f$ is a uniform function. Then $|f|\geq{2^{n-j}(q-1)^{n-j}q^{i+j-n}}$ due to Theorem~\ref{Theorem for uniform}. Since $q>3$, we obtain $|f|>2^{i}(q-1)^{n-j}$.

Thus, if $f_{k}^{v}\not\equiv0$ for any $k\in{\Sigma_q}$ and $v\in\{1,2,\ldots,n\}$ and $q\geq 5$, then $|f|>2^{i}(q-1)^{n-j}$. In particular, in this case $|f|>2^{i}(q-1)^{n-i}$ for $i=j$ and $q\geq 5$.
\end{proof}

\begin{corollary}\label{Corollary 2.1}
Let  $f\in{U_{i}(n,q)}$, $i>\lfloor\frac{n}{2}\rfloor$, $q\geq 4$ and $f\not\equiv 0$. Then $|f|\geq 2^{i}(q-1)^{n-i}$. Moreover, for $q\geq 5$ the equality $|f|=2^{i}(q-1)^{n-i}$ holds if and only if $f_{\sigma}\in F_{2}(n,q,i,i)$ for some permutation $\sigma\in{S_{n}}$.
\end{corollary}

\section{Discussion}

The initial problem of finding functions from  ${U_{[i,j]}(n,q)}$ with minimum size of the support is formulated for arbitrary real-valued functions from corresponding subspace. Surprisingly, Theorems \ref{Theorem 1.1} and \ref{Theorem 2.1} imply that such functions take only $3$ distinct values. Moreover, such functions are equal to a tensor product of several {\it elementary} eigenfunctions of the Hamming graphs of dimensions not greater that $2$ after some permutation of coordinate positions. These elementary functions belong to
$A_1\cup A_3\cup A_4$ and $A_1\cup A_2\cup A_4$ for the cases $i+j\le n$ and $i+j>n$ respectively.

One may notice, that bounds for the size of a support and corresponding characterizations obtained in Theorems \ref{Theorem 1.1} and \ref{Theorem 2.1} require some lower bounds for $q$, and in the case $i+j>n$ for $i\neq j$ there is no characterization. Further we provide several examples explaining difficulties of characterisation for the case $i+j>n$ and for small values of $q$.

\textbf{Remark 1.}
In Theorem~\ref{Theorem 2.1} we prove that $|f|\geq 2^{i}(q-1)^{n-j}$ for $f\in{U_{[i,j]}(n,q)}$, $i+j>n$ and $q\geq4$. On the other hand, if $f\in F_{2}(n,q,i,j)$ and $i+j>n$, then $|f|=2^{i}(q-1)^{n-j}$ and $f\in{U_{[i,j]}(n,q)}$ due to Lemma~\ref{constructions of eigenfunctions}. We note that in general case for $f\in{U_{[i,j]}(n,q)}$ and $i+j>n$ the equality $|f|=2^{i}(q-1)^{n-j}$ does not imply that $f_{\sigma}\in F_{2}(n,q,i,j)$ for some permutation $\sigma\in{S_{n}}$. Consider the following example:

\textbf{Example.}
We define the function $g:\Sigma_{q}^2\longrightarrow{\mathbb{R}}$ by the following rule:
$$
g(x,y)=\begin{cases}
1,&\text{if $x=y=0$;}\\
-1,&\text{if $x=y=q-1$;}\\
0,&\text{otherwise.}
\end{cases}
$$
Denote $g'(x,y)=g(y,x)$. We see that $|g|=2$. We note that $g=a_{2}(0,q-1)\cdot a_{4}(0)+a_{4}(q-1)\cdot a_{2}(0,q-1)$. Consequently $g\in{U_{[1,2]}(2,q)}$ due to Corollary~\ref{product of functions}. Thus $g(x,y)$ has the minimum size of the support in ${U_{[1,2]}(2,q)}$ but $g\not\in{F_2(2,q,1,2)}$ and $g'\not\in{F_2(2,q,1,2)}$.
Similar function can be also constructed for arbitrary $n>2$. Therefore, a possible characterization of functions from $U_{[i,j]}(n,q)$ for $i+j>n$ and $i\neq j$ in terms of tensor products of some elementary functions may contain an infinite set of different elementary functions.

\textbf{Remark 2.}
By the Corollary~\ref{Corollary 2.1} for $f\in{U_{i}(n,q)}$, $i>\lfloor\frac{n}{2}\rfloor$ and  $q\geq5$ the equality $|f|=2^{i}(q-1)^{n-i}$ holds if and only if $f_{\sigma}\in F_{2}(n,q,i,i)$ for some permutation $\sigma\in{S_{n}}$.
The following example shows that for $f\in{U_{i}(n,q)}$, $i>\lfloor\frac{n}{2}\rfloor$ and  $q=4$ the equality $|f|=2^{i}(q-1)^{n-i}$ does not imply that $f_{\sigma}\in F_{2}(n,q,i,i)$ for some permutation $\sigma\in{S_{n}}$.

\textbf{Example.}
We define the functions $h_1,h_2:\Sigma_{4}^2\longrightarrow{\mathbb{R}}$ by the following rules:
$$
h_1(x,y)=\begin{cases}
-1,&\text{if $x=y=0$;}\\
1,&\text{if $x=y=2$;}\\
0,&\text{otherwise}
\end{cases}
$$ and

$$
h_2(x,y)=\begin{cases}
1,&\text{if $x=0$ and $y\in\{1,3\}$;}\\
-1,&\text{if $y=2$ and $x\in\{1,3\}$;}\\
0,&\text{otherwise.}
\end{cases}
$$

We define the function $h:\Sigma_{4}^3\longrightarrow{\mathbb{R}}$ by the following rule:
$$
h(x,y,z)=\begin{cases}
h_1(x,y),&\text{if $z=0$ or $z=1$;}\\
h_2(x,y),&\text{if $z=2$;}\\
h_2(y,x),&\text{if $z=3$.}
\end{cases}
$$
We note that $|h|=12$. By the definition of an eigenfunction one can check that $h\in{U_2(3,4)}$.  Thus $h$ has the minimum size of the support in ${U_{2}(3,4)}$ but $h_{\sigma}\not\in{F_2(3,4,2,2)}$ for any permutation $\sigma\in S_3$.

\textbf{Remark 3.}
We note that Theorem~\ref{Theorem 2.1} does not hold for $q=3$. Let us consider the following example for $n=3$ and $i=j=2$.

\textbf{Example.}
We define the function $v_1:\Sigma_{3}^2\longrightarrow{\mathbb{R}}$ by the following rule:
$$
v_1(x,y)=\begin{cases}
1,&\text{if $x=y=0$;}\\
-1,&\text{if $x=1$ and $y=2$;}\\
0,&\text{otherwise.}
\end{cases}
$$

For $a,b\in{\Sigma_{3}}$ denote by $a\oplus b$ the sum of $a$ and $b$ modulo $3$.
We define the function $v:\Sigma_{3}^3\longrightarrow{\mathbb{R}}$ by the following rule:
$$
v(x,y,z)=\begin{cases}
v_1(x,y),&\text{if $z=0$;}\\
v_1(x\oplus1,y\oplus1),&\text{if $z=1$;}\\
v_1(x\oplus2,y\oplus2),&\text{if $z=2$.}
\end{cases}
$$
We note that $|v|=6$. By the definition of an eigenfunction one can check that $v\in{U_2(3,3)}$. So $|v|<8$ and Theorem~\ref{Theorem 2.1} does not hold in this case.

\section{Acknowledgements}
We are grateful to the referees for useful remarks.

\end{document}